\documentclass{article}

\usepackage[utf8]{inputenc}
\usepackage{textalpha} %to write greek letters in text mode
\usepackage{amsmath,amsthm,amssymb, latexsym}
\usepackage{ebproof} %for displaying formal proofs
\usepackage[hidelinks]{hyperref}
\usepackage{xcolor} %for colours
\usepackage{xspace} % for defining commands with automatic space at the end
\usepackage{xargs} % for commands with more than one optional argument
\usepackage{mathdots} %for inverse diagonal dots
\usepackage{titlesec} %for changing style of section/subsection/... titles
\theoremstyle{plain}
\newtheorem{theorem}{Theorem}
\newtheorem{lemma}[theorem]{Lemma}

\newtheorem{proposition}[theorem]{Proposition}
\theoremstyle{definition}
\newtheorem{definition}[theorem]{Definition}

\newtheorem{remark}[theorem]{Remark}

\titleformat{\section}
{\normalfont\large\bfseries}{\thesection}{1em}{}
\titleformat{\subsection}
{\normalfont\normalsize\bfseries}{\thesubsection}{1em}{}

\usepackage{fixme}
\fxsetup{
	status=draft,
	author=,
	layout=footnote,
	theme=color
}

\newcommand{\calT}{\mathcal{T}}

\newcommand{\sfN}{\mathsf{N}}

\renewcommand{\|}{~|~}

\newcommand{\mybar}[1]{\overline{#1}}

\renewcommand{\land}{\wedge}
\renewcommand{\lor}{\vee}

\newcommand{\Var}{\mathsf{Var}}

\newcommand{\AxLit}{\ensuremath{\mathsf{Ax}}\xspace}

\newcommand{\RuOr}{\ensuremath{\mathsf{R}_{\lor}}\xspace}
\newcommand{\RuAnd}{\ensuremath{\mathsf{R}_{\land}}\xspace}
\newcommand{\RuBox}{\ensuremath{\mathsf{\mathsf{R}_{\Box}}}\xspace}

\newcommand{\RuFp}[1]{\ensuremath{\mathsf{R}_{#1}}\xspace}

\newcommand{\RuNu}{\RuFp{\nu}}

\newcommand{\RuEta}{\RuFp{\eta}}
\newcommand{\RuWeak}{\ensuremath{\mathsf{R}_{\mathsf{w}}}\xspace}
\newcommand{\RuExp}{\ensuremath{\mathsf{exp}}\xspace}
\newcommand{\RuClo}[1][\dx]{\ensuremath{\nu\mathsf{-clo}_{#1}}\xspace}

\newcommand{\RuCut}{\ensuremath{\mathsf{cut}}\xspace}
\newcommand{\RuInd}{\ensuremath{\mathsf{ind}}\xspace}
\newcommand{\dx}{\ensuremath{\mathsf{x}}}
\newcommand{\dy}{\ensuremath{\mathsf{y}}}

\newcommand{\Clo}{\ensuremath{\mathsf{Clo}}\xspace}
\newcommand{\CloCut}{\ensuremath{\mathsf{Clo}+\mathsf{cut}}\xspace}

\newcommand{\NW}{\ensuremath{\mathsf{NW}}\xspace}

\newcommand{\dia}{\Diamond}

\newcommand{\atneg}[1]{\overline{#1}}
\newcommand{\isbnf}{\;::=\;}
\newcommand{\divbnf}{\;\mid\;}

\newcommand{\Prop}{\mathsf{Prop}}

\newcommand{\Clos}{\mathsf{Clos}}

\renewcommand{\phi}{\varphi}

\newcommand{\tracestep}{\to_{C}}

\title{A note on the incompleteness of Afshari \& Leigh's system \Clo}
\author{Johannes Kloibhofer}

\begin{document}
	
\maketitle

\begin{abstract}
	The system \Clo is a cyclic, cut-free proof system for the modal $\mu$-calculus. 
	It was introduced by Afshari \& Leigh as an intermediate system in their intent to show the completeness of Kozen's axiomatisation for the modal $\mu$-calculus. 
	We prove that \Clo is incomplete by giving a valid sequent $\Phi$ that is not provable in \Clo.  
\end{abstract}	

\section{Introduction}	
The proof system \Clo was introduced by Afshari \& Leigh in \cite{Afshari2017} as a cyclic, cut-free proof system for the modal $\mu$-calculus. 
In that paper they intend to prove the completeness of Kozen's original axiomatization \cite{Kozen1983} for the modal $\mu$-calculus in a proof-theoretic way.
This is done by a series of translations starting from Jungteerapanich and Stirling's \cite{Jungteerapanich2010,Stirling2014} proof system with names to \Clo and further to Kozen's system. Apart from its prominent role in this completeness proof \Clo has also attracted interest as a stand-alone proof system: It is cut-free, cyclic and has very simple annotations and discharge condition. As such the completeness of \Clo also played a crucial role in showing completeness of the natural axiomatization for game logic \cite{Enqvist2019}. 

Our contribution is to show that \Clo is in fact incomplete. This breaks the completeness proof for Kozen's system in \cite{Afshari2017} and thus leaves some further questions open: 
\begin{enumerate}
	\item Can the completeness of Kozen's axiomatization be proven without reducing the problem to the aconjunctive fragment as in Walukiewicz's proof \cite{Walukiewicz2000}?
	\item Can the completeness of Kozen's axiomatization be proven proof-theoretically?
	\item Is there a finitary, well-founded, cut-free and complete proof system for the modal $\mu$-calculus? A candidate would be the system $\mathsf{Koz}_{\mathsf{s}}^-$ introduced in \cite{Afshari2017}, for which the completeness is unknown.
	
	\item Is the natural axiomatization for game logic complete?
\end{enumerate}

\section{Preliminaries}
We assume familiarity with the modal $\mu$-calculus. In this section we fix notations for the modal $\mu$-calculus and the versions of the proof systems \NW and \Clo we are using.

\subsection{Modal $\mu$-calculus}
\emph{Formulas} of the modal $\mu$-calculus are generated 
by the grammar
\[
\phi \isbnf 
p \divbnf \atneg p 
\divbnf x
\divbnf \phi\lor\phi \divbnf \phi\land\phi \divbnf
\dia\phi \divbnf \Box\phi 
\divbnf \mu x . \phi \divbnf \nu x . \phi,
\]
where $p$ is taken from a set $\Prop$ of propositional
variables and $x$ from a set $\Var$ of formal variables. \emph{Fixpoint formulas} of the form $\mu x. \phi$ and $\nu x.\phi$ are called $\mu$-formulas and $\nu$-formulas, respectively. We use $\eta$ to denote either $\mu$ or $\nu$. We call a formula \emph{closed}, if all occurrences of formal variables are bound by a fixpoint. 

Let $\Gamma$ be a sequent, i.e., a set of closed $\mu$-calculus formulas. We assume that formulas in $\Gamma$ are \emph{clean}: With every formal variable $x$ we may associate a unique formula of the form $\eta x. \phi$.  Let $\leq$ denote the \emph{subsumption order} on the formal variables occurring in $\Gamma$, i.e. $x \leq y$ if $x$ occurs free in $\eta y. \psi$. For a fixpoint formula $\eta x. \phi$ we write $\phi[\eta x. \phi]$ for its \emph{unfolding}, i.e. the formula obtained from $\phi$ by replacing every occurrence of $x$ with $\eta x. \phi$. For a sequent $\Gamma$ we define $\Diamond \Gamma$ to be the sequent $\{\Diamond \phi \| \phi \in \Gamma\}$.

\subsection{\NW proofs}
The cyclic proof system \NW is based on the tableaux games introduced by Niwiński \& Walukiewicz \cite{Niwinski1996}. The presentation as a proof system can be justified by identifying winning strategies for one of the players in their games with proofs in the system \NW.

\begin{figure}[htb]
	\begin{align*}
		\begin{minipage}{0.28\textwidth}
			\begin{prooftree}
				\infer[left label=\AxLit:]0{ p^{},\bar{p}^{}}
			\end{prooftree}
		\end{minipage}
		\begin{minipage}{0.31\textwidth}
			\begin{prooftree}
				\hypo{ \Gamma, \varphi^{},\psi^{}}
				\infer[left label= \RuOr:]1{ \Gamma, \varphi \lor \psi^{}}
			\end{prooftree}
		\end{minipage}
		\begin{minipage}{0.32\textwidth}
			\begin{prooftree}
				\hypo{ \Gamma, \varphi^{} }
				\hypo{\Gamma, \psi^{}}
				\infer[left label= \RuAnd:]2{ \Gamma, \varphi \land \psi^{}}
			\end{prooftree}
		\end{minipage}\\
	\phantom{XXXX}\\
		\begin{minipage}{0.28\textwidth}
			\begin{prooftree}
				\hypo{ \Gamma}
				\infer[left label=\RuWeak:]1{ \Gamma, \varphi^{}}
			\end{prooftree}
		\end{minipage}
		\begin{minipage}{0.31\textwidth}
			\begin{prooftree}
				\hypo{ \Gamma, \varphi^{}}
				\infer[left label= \RuBox:]1{ \Diamond \Gamma, \Box \varphi^{}}
			\end{prooftree}
		\end{minipage}
		\begin{minipage}{0.32\textwidth}
			\begin{prooftree}
				\hypo{ \Gamma, \varphi[\eta x.\varphi]^{}}
				\infer[left label=\RuEta:]1{ \Gamma, \eta x.\varphi^{}}
			\end{prooftree}
		\end{minipage}
	\end{align*}
	
	\caption{Rules of \NW}
	\label{fig.NW}
\end{figure}

\emph{\NW derivations} are (possibly infinite) trees respecting the rules of Figure \ref{fig.NW}, where all leaves are labelled by \AxLit.

Let $\pi$ be an \NW derivation and $\tau$ be an \emph{infinite branch} of $\pi$. A \emph{trace} on $\tau= (u_i)_{i \in \omega}$ is a sequence $(\phi_i)_{i\in \omega}$ such that $\phi_i$ is a formula at the node $u_i$ and $\phi_{i+1}$ is a descendant of $\phi_i$ for $i \in \omega$, where descendants are defined as usual. 
E.g., if \RuOr is the applied rule, then the descendants of the principal formula $\phi \lor \psi$ are $\phi$ and $\psi$ and the descendant of any side formula $\chi \in \Gamma$ is $\chi$.
On every trace there is a unique fixpoint-formula $\phi$ occurring infinitely often, which is minimal with respect to the subsumption order. We call a trace \emph{$\nu$-trace} if this formula $\phi$ is a $\nu$-formula and \emph{$\mu$-trace} if it is a $\mu$-formula.

\begin{definition}
	Let \NW be the proof system defined by the rules in Figure \ref{fig.NW}. An \emph{\NW proof} $\pi$ is an \NW derivation, where on every infinite branch of $\pi$ there is a $\nu$-trace.
\end{definition}

\begin{theorem}[Niwiński \& Walukiewicz, \cite{Niwinski1996}]
	\NW is sound and complete.
\end{theorem}

\subsection{\Clo proofs}
In \Clo proofs annotations keep track of $\nu$-traces. In that sense the trace condition is replaced by local conditions on the annotations and \Clo proofs become finite. 

For each formal variable $x \in \Var$ fix an infinite set $\sfN_x$ of \emph{names} for $x$. We assume $\sfN_x \cap \sfN_y = \varnothing$ if $x \neq y$ and use symbols $\dx$ and $\dy$ as names for the formal variables $x$, $y$.
Let $\sfN := \bigcup_{x \in \Var} \sfN_x$. The subsumption order extends to names: 
$\dy \leq x$ if $\dy \in \sfN_y$ and $y \leq x$. For a finite word $a \in \sfN^*$ we let $a \leq x$ if $\dy \leq x$ for all $\dy \in a$. Let $\sqsubseteq$ denote the (reflexive) sub-word relation on $\sfN^*$.

An \emph{annotated formula} is a pair $(\phi,a)$, denoted by $\phi^a$, where $\phi$ is a closed formula and $a \in \sfN^*$. An \emph{annotated sequent} is a set of annotated formulas.

\begin{figure}[htb]
	\begin{align*}
		\begin{minipage}{0.28\textwidth}
			\begin{prooftree}
				\infer[left label=\AxLit:]0{ p^{\epsilon},\bar{p}^{\epsilon}}
			\end{prooftree}
		\end{minipage}
		\begin{minipage}{0.31\textwidth}
			\begin{prooftree}
				\hypo{ \Gamma, \varphi^{a},\psi^{a}}
				\infer[left label= \RuOr:]1{ \Gamma, (\varphi \lor \psi)^{a}}
			\end{prooftree}
		\end{minipage}
		\begin{minipage}{0.32\textwidth}
			\begin{prooftree}
				\hypo{ \Gamma, \varphi^{a} }
				\hypo{\Gamma, \psi^{a}}
				\infer[left label= \RuAnd:]2{ \Gamma, (\varphi \land \psi)^{a}}
			\end{prooftree}
		\end{minipage}\\
	\phantom{XXX}\\
		\begin{minipage}{0.28\textwidth}
			\begin{prooftree}
				\hypo{ \Gamma}
				\infer[left label=\RuWeak:]1{ \Gamma, \varphi^{a}}
			\end{prooftree}
		\end{minipage}
		\begin{minipage}{0.31\textwidth}
			\begin{prooftree}
				\hypo{ \Gamma, \varphi^{a}}
				\infer[left label= \RuBox:]1{ \Diamond \Gamma, \Box \varphi^{a}}
			\end{prooftree}
		\end{minipage}
		\begin{minipage}{0.32\textwidth}
			\begin{prooftree}
				\hypo{ \Gamma, \varphi[\eta x.\varphi]^{a}}
				\infer[left label=\RuEta:]1[~where $a \leq x$]{ \Gamma, \eta x.\varphi^{a}}
			\end{prooftree}
		\end{minipage}
	\\
	\phantom{XXX}\\
		\begin{minipage}{0.80\textwidth}
			\begin{prooftree}
				\hypo{ \varphi_1^{a_1},...,\varphi_n^{a_n}}
				\infer[left label=\RuExp:]1[~ where $a_i \sqsubseteq b_i$ for $i = 1,...,n$]{ 	\varphi_1^{b_1},...,\varphi_n^{b_n}}
			\end{prooftree}
		\end{minipage}\\
	\phantom{XXX}\\
		\begin{minipage}{0.80\textwidth}
			\begin{prooftree}
				\hypo{[\Gamma, \nu x. \varphi^{a\dx}]^{\dx}}
				\infer[no rule]1{{\vdots}}
				\infer[no rule]1{\Gamma, \varphi[\nu x. \varphi]^{a\dx}}
				\infer[left label=\RuClo:]1[~~ where $a \leq \dx$ and $\dx$ does not appear in $\Gamma$]{\Gamma, \nu x. \varphi^a}
			\end{prooftree}
		\end{minipage}
	\end{align*}
		
	\caption{Rules of \Clo}
	\label{fig.Clo}
\end{figure}	

\emph{\Clo derivations} are (possibly infinite) trees respecting the rules of Figure \ref{fig.Clo}, where all leaves $u$ are either labelled by \AxLit or by a discharge token $\dx$, such that an ancestor $v$ of $u$ is labelled by \RuClo[]. In the latter case we call $u$ a \emph{discharged assumption} and $v$ its \emph{companion} and say that $u$ is \emph{discharged} by \RuClo[] at $v$. All discharge rules \RuClo are labelled by a unique discharge token taken from the set $\sfN$, such that $\dx \in \sfN_x$ if the principal formula is $\nu x. \phi$. 
Note that, if one ignores the annotations, the rules \AxLit, \RuOr, \RuAnd, \RuWeak, \RuBox and \RuEta are the same as in \NW.

\begin{definition}
	Let \Clo be the proof system defined by the rules in Figure \ref{fig.Clo}. 	A \emph{\Clo proof} is a finite \Clo derivation. 
\end{definition}
%Our presentation of the proof system \Clo is almost identical to the one in \cite{Afshari2017}. 

\begin{proposition}
	\Clo is sound.
\end{proposition}
\begin{proof}
	Let $\rho$ be a \Clo proof of a sequent $\Gamma$. Let $\rho^*$ be the infinite unfolding of $\rho$, i.e., the \Clo derivation defined from $\rho$ by inductively replacing discharged assumptions by the subproof rooted at its companion. Replacing \RuClo[] rules by \RuNu, removing nodes labelled by \RuExp and removing all annotations in $\rho^*$ yields an \NW derivation $\pi$ of $\Gamma$. The names occurring in $\rho^*$ give a $\nu$-trace for every infinite path in $\pi$ and thus $\pi$ is an \NW proof. Hence the soundness of \Clo follows from the soundness of \NW.
\end{proof}

\section{Incompleteness of \Clo}
In the main section of this work we prove the incompleteness of \Clo. This is done by defining a sequent $\Phi$, which is probable in \NW, but shown to be unprovable in \Clo.

\subsection{Definition of $\Phi$}
Consider the formula \[\nu x. \varphi \equiv \nu x. \Diamond(\bar{p}\land(\Box x  \lor \Diamond \nu y. \Box(p \land (\Box x \lor \Diamond y)))).\]
As abbreviation let  
\[\nu y. \psi \equiv \nu y. \Box(p \land (\Box \nu x. \varphi \lor \Diamond y)))\]
and
\[ \chi \equiv \Box \nu x. \varphi \lor \Diamond \nu y. \psi.\]
In order to understand the relation between these formulas, observe that $\phi[\nu x.\phi] \equiv \Diamond(\bar{p} \land \chi)$ and $\psi[\nu y. \psi] \equiv \Box (p \land \chi)$. The reason for the excessive use of modalities in the definition of $\nu x. \phi$ is to restrict possibilities in the proof search: For most sequents only one rule apart from weakenings will be applicable.

For the rest of the paper let $\Phi = \nu x. \phi, \nu y. \psi$. We want to show that $\bigvee \Phi$ is a valid $\mu$-calculus formula, yet unprovable in \Clo.

\subsection{Validity}
The following is an \NW proof $\pi$ of $\Phi$, where the subproofs $\pi'$ are isomorphic to the whole proof $\pi$:

\begin{align*}
	\begin{prooftree}
		\hypo{\bar{p},p}
		\hypo{\pi'}
		\infer[no rule]1{\nu x.\phi,\nu y. \psi}
		\infer[left label=B:]1[\RuBox]{\Box \nu x.\phi,\Diamond \nu y. \psi}
		\infer1[\RuWeak]{\Box \nu x.\phi,\Diamond \nu y. \psi, p}
		\infer1[\RuOr]{\chi,p}
		\infer2[\RuAnd]{\bar{p}\land \chi,p}
		\hypo{\pi'}
		\infer[no rule]1{\nu x.\phi,\nu y. \psi}
		\infer[left label=C:]1[\RuBox]{\Box \nu x.\phi,\Diamond \nu y. \psi}
		\infer1[\RuWeak]{\bar{p}, \Box \nu x.\phi,\Diamond \nu y. \psi}
		\infer1[\RuOr]{\bar{p},\chi}
		\hypo{\pi'}
		\infer[no rule]1{\nu x.\phi,\nu y. \psi}
		\infer[left label=]1[\RuBox]{\Box \nu x.\phi,\Diamond \nu y. \psi}
		\infer1[\RuOr]{\chi}
		\infer2[\RuAnd]{\bar{p}\land \chi,\chi}
		\infer2[\RuAnd]{\bar{p}\land\chi, p \land \chi}
		\infer1[\RuBox]{\phi[\nu x.\phi],\psi[\nu y. \psi]}
		\infer1[\RuNu]{\phi[\nu x.\phi],\nu y. \psi}
		\infer1[\RuNu]{\nu x.\phi,\nu y. \psi}
	\end{prooftree}
\end{align*}

As there are no $\mu$-formulas, all traces in $\pi$ are $\nu$-traces. Hence on every infinite branch of $\pi$ there is a $\nu$-trace and $\pi$ is an \NW proof. Using the fact that \NW is a sound proof system it follows that $\bigvee \Phi$ is a valid $\mu$-calculus formula.

\subsection{Proof idea}\label{subsec.proofIdea}
The more difficult task is to show that $\Phi$ is not provable in \Clo. To gather some intuition we first consider the above \NW proof $\pi$ and see why it can not be translated into a \Clo proof $\rho$. At the node $B$ both formulas are descendants of $\nu x.\phi$ at the root, whereas at the node $C$ both formulas are descendants of $\nu y.\psi$ at the root. Replacing the \RuNu rules by \RuClo[] rules yields the following \Clo derivation $\rho_0$:\footnote{We omit the rightmost branch, as it is not important for this example.}

\begin{align*}
	\begin{prooftree}
		\hypo{\bar{p}^\dx,p^\dy}
		\hypo{[\nu x.\phi^\dx,\nu y. \psi]^\dx}
		\infer[left label=B:]1[\RuExp]{\nu x.\phi^\dx,\nu y. \psi^\dx}
		\infer1[\RuWeak,\RuBox]{\Box \nu x.\phi^\dx,\Diamond \nu y. \psi^\dx, p^\dy}
		\infer1[\RuOr]{\chi^\dx,p^\dy}
		\infer2[\RuAnd]{\bar{p}\land \chi^\dx,p^\dy}
		\hypo{\phi[\nu x.\phi],\nu y. \psi^\dy}
		\infer[left label=C:]1[\RuExp]{\phi[\nu x.\phi]^\dy,\nu y. \psi^\dy}
		\infer1[\RuNu]{\nu x.\phi^\dy,\nu y. \psi^\dy}
		\infer1[\RuWeak,\RuBox]{\bar{p}^\dx,\Box \nu x.\phi^\dy,\Diamond \nu y. \psi^\dy}
		\infer1[\RuOr]{\bar{p}^\dx,\chi^\dy}
		\hypo{\cdots}
		\infer2[\RuAnd]{\bar{p}\land \chi^\dx,\chi^\dy}
		\infer2[\RuAnd]{\bar{p}\land\chi^\dx, p \land \chi^\dy}
		\infer1[\RuBox]{\phi[\nu x.\phi]^\dx,\psi[\nu y. \psi]^\dy}
		\infer1[\RuClo[\dy]]{\phi[\nu x.\phi]^\dx,\nu y. \psi}
		\infer1[\RuClo]{\nu x.\phi,\nu y. \psi}
	\end{prooftree}
\end{align*}
The leaf B can be discharged by \RuClo, yet it is impossible to discharge C by \RuClo[\dy], as $\phi[\nu x.\phi]$ is annotated by $\dx$ at the companion node and there is no way to obtain the same annotation at C. Thus we can not discharge C by \RuClo[\dx] nor by \RuClo[\dy].  We will now see that the same problem occurs in all \Clo derivations of $\Phi$.

As \Clo is cut-free we can do proof search adapted for cyclic, annotated proofs in the following way: First we consider all \NW proofs of $\Phi$. Then we show that in all of those \NW proofs it is impossible to replace some \RuNu rules by \RuClo[] rules with discharged assumptions in a way that a \Clo proof is obtained.

\subsection{\NW proofs of $\Phi$}

Let $\Phi = \nu x. \phi,\nu y. \psi$. We want to consider all possible \NW proofs of $\Phi$. Thus let $\pi$ be any \NW proof of $\Phi$ and $r$ be any node in $\pi$ labelled by the sequent $\Phi$. We want to have a look at the subtree rooted at $r$, where leaves are axioms or nodes labelled by $\Phi$.

We begin by claiming that the first two applied rules at $r$ have to be \RuNu rules with respective principal formulas $\nu x. \phi$ and $\nu y. \psi$, in no particular order. 
The only other rules which may be applied are instances of \RuWeak. Yet this is impossible, as $\nu x. \phi$ and $\nu y. \psi$ are not valid, thus not provable. 

Hence the first two applied rules are \RuNu rules; for now the order in which the formulas are unfolded first is not important. The lowest part of the proof-tree rooted at $r$ looks as follows:

\begin{align*}
	\begin{prooftree}
		\hypo{\bar{p}\land \chi, p \land \chi}
		\infer1[\RuBox]{\phi[\nu x.\phi],\psi[\nu y. \psi]}
		\infer1[\RuNu]{\phi[\nu x.\phi],\nu y. \psi}
		\infer1[\RuNu]{\nu x.\phi,\nu y. \psi}
	\end{prooftree}
\end{align*}

Assume that \RuWeak is only applied for inessential literals, then the proof tree rooted at $r$ looks as follows, up to the order of the applied rules:

\begin{align*}
	\begin{prooftree}
		\hypo{\bar{p},p}
		\hypo{\nu x.\phi,\nu y. \psi}
		\infer[left label=B:]1[\RuBox]{\Box \nu x.\phi,\Diamond \nu y. \psi} 
		\infer1[\RuWeak]{\Box \nu x.\phi,\Diamond \nu y. \psi, p} 
		\infer1[\RuOr]{\chi,p}
		\infer[left label=A:]2[\RuAnd]{\bar{p}\land \chi,p}
		\hypo{\nu x.\phi,\nu y. \psi}
		\infer[left label=C:]1[\RuBox]{\Box \nu x.\phi,\Diamond \nu y. \psi}
		\infer1[\RuWeak]{\bar{p},\Box \nu x.\phi,\Diamond \nu y. \psi}
		\infer1[\RuOr]{\bar{p},\chi}
		\hypo{\nu x.\phi,\nu y. \psi}
		\infer[left label=D:]1[\RuBox]{\Box \nu x.\phi,\Diamond \nu y. \psi}
		\infer1[\RuOr]{\chi}
		\infer2[\RuAnd]{\bar{p}\land \chi,\chi}
		\infer2[\RuAnd]{\bar{p}\land \chi, p \land \chi}
		\infer1[\RuBox]{\phi[\nu x.\phi],\psi[\nu y. \psi]}
		\infer[]1[\RuNu]{\phi[\nu x.\phi],\nu y. \psi}
		\infer[]1[\RuNu]{\nu x.\phi,\nu y. \psi}
	\end{prooftree}
\end{align*}
Note that, in whatever order the rules are applied, we end up with the same four nodes A,B,C and D.

At last we have to check the case, that an instance of \RuWeak is applied to a non-literal. Leaving those weakenings aside which lead to invalid formulas, there are only two possible \NW proofs (up to the order of the unfoldings):
\begin{align*}
	\begin{prooftree}
		\hypo{\bar{p},p}
		\hypo{\nu x.\phi,\nu y. \psi}
		\infer[left label=B:]1[\RuBox]{\Box \nu x.\phi,\Diamond \nu y. \psi}
		\infer1[\RuWeak]{\Box \nu x.\phi,\Diamond \nu y. \psi, p}
		\infer1[\RuOr]{\chi,p}
		\infer[left label=A:]2[\RuAnd]{\bar{p}\land \chi,p}
		\hypo{\nu x.\phi,\nu y. \psi}
		\infer[left label=C:]1[\RuBox]{\Box \nu x.\phi,\Diamond \nu y. \psi}
		\infer1[\RuOr]{\chi}
		\infer1[\RuWeak]{\bar{p}\land \chi,\chi}
		\infer2[\RuAnd]{\bar{p}\land \chi, p \land \chi}
		\infer1[\RuBox]{\phi[\nu x.\phi],\psi[\nu y. \psi]}
		\infer1[\RuNu]{\phi[\nu x.\phi],\nu y. \psi}
		\infer1[\RuNu]{\nu x.\phi,\nu y. \psi}
	\end{prooftree}
\end{align*}

\begin{align*}
	\begin{prooftree}
		\hypo{\bar{p},p}
		\hypo{\nu x.\phi,\nu y. \psi}
		\infer[left label=C:]1[\RuBox]{\Box \nu x.\phi,\Diamond \nu y. \psi}
		\infer1[\RuWeak]{\bar{p},\Box \nu x.\phi,\Diamond \nu y. \psi}
		\infer1[\RuOr]{\bar{p},\chi}
		\infer[left label=A:]2[\RuAnd]{\bar{p},p \land \chi}
		\hypo{\nu x.\phi,\nu y. \psi}
		\infer[left label=B:]1[\RuBox]{\Box \nu x.\phi,\Diamond \nu y. \psi}
		\infer1[\RuOr]{\chi}
		\infer1[\RuWeak]{\chi,p \land \chi}
		\infer2[\RuAnd]{\bar{p}\land \chi, p \land \chi}
		\infer1[\RuBox]{\phi[\nu x.\phi],\psi[\nu y. \psi]}
		\infer[]1[\RuNu]{\phi[\nu x.\phi],\nu y. \psi}
		\infer[]1[\RuNu]{\nu x.\phi,\nu y. \psi}
	\end{prooftree}
\end{align*}
Now we have considered all possible \NW proofs rooted at a node labelled by $\Phi$, where leaves are axioms or nodes labelled by $\Phi$. We can state some immediate observations.

\medskip
\noindent
Let $\pi$ be an \NW proof of $\Phi$.
	\begin{enumerate}
		\item We call a node which is labelled by the sequent $\Phi$ an \emph{unfolding node}. An unfolding node and its child are always labelled by \RuNu rules with principal formulas $\nu x. \phi$ and $\nu y. \psi$, in no particular order. 
		\item The \emph{unfolding tree} $\calT_\pi = (T,E)$ of $\pi$ is the tree consisting of all unfolding nodes in $\pi$, such that $E(u,v)$ if $v$ is a descendant of $u$ with no other unfolding nodes between $u$ and $v$.
		\item Let $v$ be an unfolding node and $u \in \calT_\pi$ its parent. We call $v$ an \emph{$x$-node} if there are no traces from $\nu y. \psi$ at $u$ to $\nu x. \phi$ or $\nu y. \psi$ at $v$. We call $v$ a \emph{$y$-node} if there are no traces from $\nu x. \phi$ at $u$ to $\nu x. \phi$ or $\nu y. \psi$ at $v$. 
	\end{enumerate}

\begin{lemma}\label{lem.preProof}
	Let $\pi$ be an \NW proof of $\Phi$. Every unfolding node of $\pi$ has either two or three children in $\calT_\pi$, where exactly one is an $x$-node and exactly one is a $y$-node.
\end{lemma}
\begin{proof}
	In the above proofs the nodes B are $x$-nodes, the nodes C are $y$-nodes and D are neither.
\end{proof}

\subsection{\Clo derivations of $\Phi$}

\begin{definition}
	Let $\pi$ be an \NW proof of a sequent $\Gamma$ and let $\rho$ be a \Clo derivation of $\Gamma$. We say that $\rho$ is \emph{obtained from $\pi$} if $\pi$ can be transformed to $\rho$ by 
	\begin{enumerate}
		\item changing some \RuNu rules to \RuClo[] rules with discharged assumptions, such that the proof tree is pruned at discharged assumptions and
		\item adding annotations and \RuExp rules accordingly.
	\end{enumerate}
\end{definition}

\Clo derivations $\rho$ obtained from an \NW proof $\pi$ have a very similar structure to the \NW proof $\pi$. Except nodes labelled by \RuExp, $\rho$ consists of the same nodes as in $\pi$, where the tree is pruned at discharged assumptions. We transfer the concepts of unfolding trees, $x$-nodes and $y$-nodes to \Clo derivations and get similar results for \Clo derivations obtained from \NW proofs.

\begin{definition}
	Let $\rho$ be a \Clo derivation of $\Phi$. Similarly as for \NW proofs we call a node in $\rho$, which is labelled by the sequent $\Phi$ and not by \RuExp, an \emph{unfolding node}.  The \emph{unfolding tree} $\calT_\rho$ of $\rho$, \emph{$x$-nodes} and \emph{$y$-nodes} are defined analogously as for \NW proofs. We call an unfolding node $u$ \emph{root-like} if no node in the subtree of $\rho$ rooted at $u$ is discharged by a \RuClo[] rule at an ancestor-node of $u$.
\end{definition}

\begin{lemma}\label{lem.obtainedXnode}
	Let $\pi$ be an \NW proof of $\Phi$ and $\rho$ be a \Clo derivation obtained from $\pi$. Every root-like unfolding node of $\rho$ has either two or three children in $\calT_\rho$, where exactly one is an $x$-node and exactly one is a $y$-node.
\end{lemma}
\begin{proof}
	Let $u$ be a root-like unfolding node in $\rho$ and $u'$ be the node in $\pi$, from which $u$ is obtained. It holds that $u'$ is an unfolding node in $\pi$. Let $v'$ be a child of $u'$ in $\calT_\pi$ and $\tau'$ be the path from $u'$ to $v'$. As no node in the subtree of $\rho$ rooted at $u$ is discharged by an ancestor-node of $u$, the path $\tau'$ has been transformed to a path $\tau$ from $u$ to $v$ in $\rho$, where only \RuNu rules were changed to \RuClo[] rules and nodes labelled by \RuExp were added. Thus $v$ is an unfolding node in $\rho$, which is an $x$-node (a $y$-node) iff $v'$ is an $x$-node (a $y$-node).

	Thus $u$ has the same number of children in $\calT_\rho$ as $u'$ has in $\calT_\pi$ and a child $v$ in $\rho$ is an $x$-node (a $y$-node) iff $v'$ in $\pi$ is an $x$-node (a $y$-node). Hence the statement follows from Lemma \ref{lem.preProof}.
\end{proof}

\begin{lemma}\label{lem.CloDerivNoDischarge}
	Let $\rho$ be a \Clo derivation obtained from an \NW proof $\pi$ of $\Phi$. Let $u$ and $v$ be unfolding nodes in $\rho$, such that $v$ is a child of $u$ in $\calT_\rho$ and $v$ is an $x$-node. Let $u'$ be between $u$ and $v$ in $\rho$, such that $u'$ is labelled by \RuClo[\dy] with principal formula $\nu y. \psi$. Then none of its discharged assumptions are in the subtree of $\rho$ rooted at $v$. 
	
	The same holds, if $v$ is a $y$-node and the principal formula at $u'$ is $\nu x. \phi$. \label{lem.replacingNuSubtree}
%	\item Let $u$ be an unfolding node and $v$ its child in $\pi$. We know that $u$ and $v$ are labelled by the rule \RuNu. Then at most one of the \RuNu rules can be replaced by an \RuClo[] rule. \label{lem.replacingNuOnlyOne}
\end{lemma}
\begin{proof}
	All nodes between $u$ and $u'$ are labelled by \RuNu, \RuClo[] or \RuExp, thus any trace from $\nu y.\phi$ at $u'$ would also give a trace from $\nu y.\phi$ at $u$. As $v$ is an $x$-node, there are no traces from $\nu y. \psi$ at $u$ to $\nu x. \phi$ or $\nu y. \psi$ at $v$ and therefore also none from $\nu y. \psi$ at $u'$. Hence the name $\dy$ introduced by the rule \RuClo[\dy] does not occur in the subtree of $\rho$ rooted at $v$.
\end{proof}

As an illustration of Lemma \ref{lem.CloDerivNoDischarge} consider the \Clo derivation $\rho_0$ presented in Subsection \ref{subsec.proofIdea}. Let $u$ be the root of $\rho_0$, $u'$ its child and $v$ be the node B. Lemma \ref{lem.CloDerivNoDischarge} states that there can be no discharged assumption of \RuClo[\dy] in the subproof rooted at B. This can be seen as the name $\dy$ does not occur at B. 

Similarly, there are no discharged assumptions of \RuClo in the subproof of $\rho_0$ rooted at C.

\subsection{\Clo proofs of $\Phi$}
We show that every \Clo proof can be obtained from some \NW proof. On the other hand we see that every \Clo derivation of $\Phi$ obtained from an \NW proof is infinite, and thus not a \Clo proof.

\begin{lemma}
	Every \Clo proof $\rho$ of a sequent $\Gamma$ can be obtained from some \NW proof $\pi$ of $\Gamma$.
\end{lemma}
\begin{proof}
	Let $\rho$ be a \Clo proof and $\rho^*$ its infinite unfolding, i.e. the \Clo derivation defined from $\rho$ by inductively replacing discharged assumptions by the subproof rooted at its companion. Replacing \RuClo[] rules by \RuNu rules in $\rho^*$, removing nodes labelled by \RuExp and removing annotations yields an \NW proof $\pi$ of $\Gamma$, from which $\rho$ can be obtained.
\end{proof}

\begin{lemma}
	There is no \Clo proof of $\Phi$.
\end{lemma}
\begin{proof}
	Let $\pi$ be an \NW derivation of $\Phi$ and $\rho$ be a \Clo derivation obtained from $\pi$. Let $\calT_{\rho}$ be the unfolding tree of $\rho$, we want to show that $\calT_{\rho}$ is infinite. This implies that $\rho$ is infinite as well and thus not a \Clo proof.  We define the \emph{height} of a node $u \in \calT_{\rho}$ to be the number of ancestors of $u$ in $\calT_{\rho}$. 
	We show by induction: 
	\begin{itemize}
		\item[$\circledast$]There exists a root-like unfolding node $u \in \calT_{\rho}$ of arbitrary height.
	\end{itemize} 
Recall that an unfolding node $u$ is called root-like, if no node in the subtree of $\rho$ rooted at $u$ is discharged by a \RuClo[] rule at an ancestor-node of $u$.
The induction base is trivial, as the root is an unfolding node, which does not have any ancestors. 
For the induction step assume that $u$ is a root-like unfolding node. The unfolding node $u$ is labelled by \RuNu or \RuClo[]. Let $u'$ be the lowest descendant of $u$ in $\rho$, which is labelled by \RuNu or \RuClo[].
 
We make a case distinction on whether $u$ and $u'$ are labelled by \RuNu or \RuClo[].
If neither $u$ nor $u'$ are labelled by \RuClo[], then any child of $u$ in $\calT_\rho$ has the same ancestors, which are labelled by \RuClo[], than $u$. Thus any child of $u$ in $\calT_{\rho}$ is root-like. 

Assume that exactly one of $u$ and $u'$, say $u_i$, is labelled by \RuClo[]; without loss of generality assume that $u_i$ is labelled by \RuClo[\dy] with principal formula $\nu y. \psi$. Let $v$ be the child of $u$ in $\calT_{\rho}$ that is an $x$-node, which exists because of Lemma \ref{lem.obtainedXnode}. Then Lemma \ref{lem.CloDerivNoDischarge} states that none of the discharged assumptions of the \RuClo[\dy] rule are in the subtree rooted at $v$. Together with the induction hypothesis this proves that $v$ is root-like. 

If both $u$ and $u'$ are labelled by \RuClo[] we may without loss of generality assume that $u$ is labelled \RuClo with principal formula $\nu x. \phi$. Then $u'$ is labelled by the sequent $\phi[\nu x. \phi]^{a\dx},\nu y. \psi^{b}$ for some annotations $a$ and $b$, and so it must be labelled by \RuClo[\dy] with principal formula $\nu y. \psi$ and premise $\phi[\nu x. \phi]^{a\dx},\psi[\nu y. \psi]^{b\dy}$. Let $v$ be the child of $u$ in $\calT_{\rho}$, that is a $y$-node. In particular there are no traces from $\phi[\nu x. \phi]$ at $u'$ to $\nu x. \phi$ or $\nu y. \psi$ at $v$. Thus the name $\dx$ is not occurring in the subtree rooted at $v$. In particular, no node in the subtree rooted at $v$ is labelled by $\phi[\nu x. \phi]^{a\dx},\nu y. \psi^{b\dy}$ and therefore there is no discharged assumption of \RuClo[\dy] in the subtree rooted at $v$. As $v$ is a $y$-node there is also no discharged assumption of \RuClo in the subtree rooted at $v$ because of Lemma \ref{lem.CloDerivNoDischarge}. Thus $v$ is root-like.

In particular $\rho$ is infinite and hence not a \Clo proof. As every \Clo proof can be obtained from an \NW proof, there is no \Clo proof of $\Phi$.
\end{proof}

\begin{theorem}
	\Clo is not complete.
\end{theorem}
\begin{proof}
	\Clo does not prove the valid sequent $\Phi$.
\end{proof}

\section{Variations of \Clo}
We mention that \Clo is complete if we add the \RuCut rule or if we restrict formulas to be adisjunctive.
\begin{remark}
	Let \CloCut be the proof system expanding \Clo by the \RuCut rule
	\begin{align*}
		\begin{prooftree}
			\hypo{\Gamma,A^\epsilon}
			\hypo{\Gamma,\mybar{A}^\epsilon}
			\infer2[\RuCut]{\Gamma}
		\end{prooftree}
	\end{align*}
	 In \CloCut one can show the admissibility of the induction rule
	\begin{align*}
		\begin{prooftree}
			\hypo{\Gamma,A[\mybar{\Gamma}]}
			\infer1[\RuInd]{\Gamma, \nu x.A}
		\end{prooftree}
	\end{align*}
	All other rules in Kozen's system are trivially admissible in \CloCut. Hence the system \CloCut is complete due to the completeness of Kozen's system. This proof presupposes the completeness proof by Walukiewicz \cite{Walukiewicz2000}.
	
	If one could show the completeness of \CloCut directly, it would yield an alternative proof for the completeness of Kozen's system, because the translation from \Clo to Kozen's system given in \cite{Afshari2017} also works with the inclusion of \RuCut.
\end{remark}

\begin{remark}
	Given two formulas $\phi,\psi$ we write $\phi \tracestep \psi$ if either $\phi$ is a boolean or modal formula and $\psi$ is a direct subformula of $\phi$ or $\phi$ is a fixpoint formula and $\psi$ is its unfolding.
	The \emph{closure} $\Clos(\Gamma)$ of a sequent $\Gamma$
	is the least superset of $\Gamma$ that is closed under this relation.
	
	We call a sequent $\Gamma$ \emph{adisjunctive} if for every fixpoint formula of the form $\nu x. \phi \in \Clos(\Gamma)$ and  $\psi_0 \lor \psi_1 \in \Clos(\nu x. \phi)$ either $\nu x. \phi \notin \Clos(\psi_0)$ or $\nu x. \phi \notin \Clos(\psi_1)$.
	
	The sequent $\Phi$ is not adisjunctive and this is crucial in showing that $\Phi$ is not provable in \Clo. Observe that $\chi \equiv \Box \nu x. \varphi \lor \Diamond \nu y. \psi \in \Clos(\nu x.\phi)$ and $\nu x.\phi \in \Clos(\Box \nu x. \varphi)$ as well as $\nu x.\phi \in \Clos(\Diamond\nu y. \psi)$.
	
	The notion of adisjunctivity is dual to that of aconjunctivity introduced by Kozen \cite{Kozen1983}. This duality stems from us proving validity in contrast to satisfiability in \cite{Kozen1983}. For the aconjunctive fragment of the modal $\mu$-calculus Kozen proved the completeness of Kozen's system in a rather direct way. 
	Similarly \Clo is complete for the adisjunctive fragment of the modal $\mu$-calculus. This can be seen as \cite[Theorem V.3.]{Afshari2017} goes through for adisjunctive sequents. 
\end{remark}

\bibliographystyle{plain}
\bibliography{Clopleteness}

\end{document}